\documentclass[10pt]{amsart}
\usepackage{amssymb,amsmath,amsfonts,latexsym,amsthm,geometry}
\usepackage{amsmath}
\usepackage{amsfonts}
\usepackage{amssymb}
\usepackage{graphicx}

\providecommand{\U}[1]{\protect\rule{.1in}{.1in}}
\providecommand{\U}[1]{\protect\rule{.1in}{.1in}}
\providecommand{\U}[1]{\protect\rule{.1in}{.1in}}
\newtheorem{theorem}{Theorem}[section]
\newtheorem{corollary}[theorem]{Corollary}

\newtheorem{lemma}[theorem]{Lemma}

\theoremstyle{definition}

\newtheorem{remark}[theorem]{Remark}
\newtheorem{definition}[theorem]{Definition}
\input{tcilatex}

\begin{document}
\title[Ideals of polynomials between Banach spaces revisited]{Ideals of
polynomials between Banach spaces revisited}
\author[T. Velanga]{T. Velanga}
\address{IMECC \\
UNICAMP-Universidade Estadual de Campinas \\
13.083-859 - S\~{a}o Paulo, Brazil.}
\address{Departamento de Matem\'{a}tica \\
Universidade Federal de Rond\^onia \\
76.801-059 - Porto Velho, Brazil.}
\email{thiagovelanga@unir.br\\
ra115476@ime.unicamp.br}
\thanks{T. Velanga was supported by FAPERO and CAPES}

\begin{abstract}
Ideals of polynomials and multilinear operators between Banach spaces have
been exhaustively investigated in the last decades. In this paper, we
introduce a unified (and more general) approach and propose some lines of
investigation in this new framework. Among other results, we prove a
Bohnenblust--Hille inequality in this more general setting.
\end{abstract}

\maketitle
\tableofcontents


\begin{center}
\textit{To the memory of Professor Jorge Mujica}
\end{center}

\section{Introduction}

Linear Functional Analysis emerged in the 30's after the publication of
Banach's monograph. The investigation of polynomials and multilinear
operators between normed spaces is, of course, the first natural step when
moving from linear to nonlinear Functional Analysis. The theory of
polynomials between normed spaces is a basic tool for the investigation of
holomorphic mappings in Banach spaces. We recall that if $E,F$ are normed
spaces, a map $P:E\rightarrow F$ is called an $m$-homogeneous polynomial
when there is an $m$-linear operator%
\begin{equation*}
A:E\times \cdots \times E\rightarrow F
\end{equation*}%
such that%
\begin{equation*}
P(x)=A(x,\ldots ,x)
\end{equation*}%
for all $x$ in $E.$ Continuity is defined as usual when dealing with metric
spaces, and it is well known that $P$ is continuous if and only if%
\begin{equation*}
\left\Vert P\right\Vert :=\sup_{\left\Vert x\right\Vert \leq 1}\left\Vert
P(x)\right\Vert <\infty .
\end{equation*}%
The basics of the theory of polynomials and multilinear operators between
Banach spaces can be found in the classical books \cite{Dineen, Muj}.
Polynomials and multilinear operators have been exhaustively investigated by
quite different viewpoints. While polynomials are suitable to investigation
of the holomorphic mappings, multilinear operators are commonly explored in
the context of the extension of the operators ideals theory to the nonlinear
setting. The notion of ideals of polynomials between Banach spaces is due to
Pietsch \cite{pi}. The natural extension to multilinear operators and
polynomials was designed by Pietsch some years later in \cite{pi2}.
Nowadays, ideals of polynomials and multilinear operators are explored by
several authors in different directions (see, for instance, \cite{ach, be,
bo01, bo2, bo4, bo3, bo9, bo10, bo8, bo6, fg}). In this paper we are mainly
interested in the theory of ideals of polynomials and ideals of multilinear
operators between Banach spaces. We propose an unified approach to the
subject and some themes for future research.

\section{Ideals of polynomials and multilinear operators: the classic
definitions}

We first recall the classical definition of operator ideals.

\begin{definition}[Operator ideal \protect\cite{pi}]
An \textit{operator ideal }is a class $\mathcal{I}$ of continuous linear
operators between Banach spaces such that for all Banach space $E$ and $F$,
its components%
\begin{equation*}
\mathcal{I}\left( E;F\right) :=\mathcal{L}\left( E;F\right) \cap \mathcal{I}
\end{equation*}%
satisfy:

\begin{description}
\item[(Oa)] $\mathcal{I}\left( E;F\right) $ is a linear subspace of $%
\mathcal{L}\left( E;F\right) $ which contains the finite rank operators;

\item[(Ob)] the ideal property: if $u\in \mathcal{I}\left( E;F\right) ,~v\in 
\mathcal{L}\left( G;E\right) $ and $t\in \mathcal{L}\left( F;H\right) $, then%
\begin{equation*}
t\circ u\circ v\in \mathcal{I}\left( G;H\right) .
\end{equation*}
\end{description}

\noindent Moreover, $\mathcal{I}$ is said to be a \textit{(quasi-) normed
operator ideal} if there exists a map \linebreak $\left\Vert \cdot
\right\Vert _{\mathcal{I}}:\mathcal{I}\rightarrow \lbrack 0,\infty )$
satisfying:

\begin{description}
\item[(O1)] $\left\Vert \cdot \right\Vert _{\mathcal{I}}$ restricted to $%
\mathcal{I}\left( E;F\right) $ is a (quasi-) norm, for all Banach spaces $E$
and $F$;

\item[(O2)] $\left\Vert id_{\mathbb{K}}:\mathbb{K}\rightarrow \mathbb{K}:id_{%
\mathbb{K}}\left( \lambda \right) =\lambda \right\Vert _{\mathcal{I}}=1$;

\item[(O3)] if if $u\in \mathcal{I}\left( E;F\right) ,~v\in \mathcal{L}%
\left( G;E\right) $ and $t\in \mathcal{L}\left( F;H\right) $, then%
\begin{equation*}
\left\Vert t\circ u\circ v\right\Vert _{\mathcal{I}}\leq \left\Vert
t\right\Vert \left\Vert u\right\Vert _{\mathcal{I}}\left\Vert v\right\Vert 
\text{.}
\end{equation*}
\end{description}

\noindent When all the components $\mathcal{I}\left( E;F\right) $ are
complete under the (quasi-) norm $\left\Vert \cdot \right\Vert _{\mathcal{I}%
} $ above, then $\mathcal{I}$ is called a \textit{(quasi-) Banach operator
ideal.}
\end{definition}

For the multilinear operators we have the following concepts:

\begin{definition}
Let $E_{1},\ldots ,E_{m},F$ be normed spaces. A multilinear mapping $T\in $ $%
\mathcal{L}\left( E_{1},\ldots ,E_{m};F\right) $ is said to be of\ \textit{%
finite type} if there exist $k\in \mathbb{N},~\varphi _{i}^{\left( j\right)
}\in E_{j}^{\prime }$ and $b_{i}\in F$ for $i=1,\ldots ,k$ and $j=1,\ldots
,m $, such that%
\begin{equation*}
T\left( x_{1},\ldots ,x_{m}\right) =\overset{k}{\underset{i=1}{\sum }}%
\varphi _{i}^{\left( 1\right) }\left( x_{1}\right) \cdots \varphi
_{i}^{\left( m\right) }\left( x_{m}\right) b_{i}\text{.}
\end{equation*}%
The subspace of all finite-type members of $\mathcal{L}\left( E_{1},\ldots
,E_{m};F\right) $ is denoted by $\mathcal{L}_{f}\left( E_{1},\ldots
,E_{m};F\right) $.
\end{definition}

\begin{definition}[Ideal of multilinear mappings \protect\cite{fg}]
\label{multilinidealdef}For each positive integer $m$, let $\mathcal{L}_{m}$
denote the class of all continuous $m$-linear operators between Banach
spaces. An \textit{ideal of multilinear mappings} $\mathcal{M}$ is a
subclass of the class $\mathcal{L}=\underset{m=1}{\overset{\infty }{\bigcup }%
}\mathcal{L}_{m}$ of all continuous multilinear operators between Banach
spaces such that for a positive integer $m$, Banach spaces $E_{1},\ldots
,E_{m}$ and $F$, the components%
\begin{equation*}
\mathcal{M}_{m}\left( E_{1},\ldots ,E_{m};F\right) :=\mathcal{L}_{m}\left(
E_{1},\ldots ,E_{m};F\right) \cap \mathcal{M}
\end{equation*}%
satisfy:

\begin{description}
\item[(Ma)] $\mathcal{M}_{m}\left( E_{1},\ldots,E_{m};F\right) $ is a linear
subspace of $\mathcal{L}_{m}\left( E_{1},\ldots,E_{m};F\right) $ which
contains the $m$-linear mappings of finite type;

\item[(Mb)] the ideal property: if $T\in \mathcal{M}_{m}\left( E_{1},\ldots
,E_{m};F\right) ,~u_{j}\in \mathcal{L}_{1}\left( G_{j};E_{j}\right) $ for $%
j=1,\ldots ,m$, and $t\in \mathcal{L}_{1}\left( F;H\right) $, then%
\begin{equation*}
t\circ T\circ \left( u_{1},\ldots ,u_{m}\right) \in \mathcal{M}_{m}\left(
G_{1},\ldots ,G_{m};H\right) .
\end{equation*}
\end{description}

\noindent Moreover, $\mathcal{M}$ is said to be a \textit{(quasi-) normed
ideal of multilinear mappings} if there exists a map $\left\Vert \cdot
\right\Vert _{\mathcal{M}}:\mathcal{M}\rightarrow \lbrack 0,\infty )$
satisfying:

\begin{description}
\item[(M1)] $\left\Vert \cdot \right\Vert _{\mathcal{M}}$ restricted to $%
\mathcal{M}_{m}\left( E_{1},\ldots ,E_{m};F\right) $ is a (quasi-) norm, for
all $m\in \mathbb{N}$ and Banach spaces $E_{1},\ldots ,E_{m}$ and $F$;

\item[(M2)] $\left\Vert id_{m}:\mathbb{K}^{m}\rightarrow \mathbb{K}%
:id_{m}\left( \lambda _{1},\ldots ,\lambda _{m}\right) =\lambda _{1}\cdots
\lambda _{m}\right\Vert _{\mathcal{M}}=1$, for all $m\in \mathbb{N}$;

\item[(M3)] If $T\in\mathcal{M}_{m}\left( E_{1},\ldots,E_{m};F\right)
,~u_{j}\in\mathcal{L}_{1}\left( G_{j};E_{j}\right) $ for $j=1,\ldots,m$, and 
$t\in\mathcal{L}_{1}\left( F;H\right) $, then%
\begin{equation*}
\left\Vert t\circ T\circ\left( u_{1},\ldots,u_{m}\right) \right\Vert _{%
\mathcal{M}}\leq\left\Vert t\right\Vert \left\Vert T\right\Vert _{\mathcal{M}%
}\left\Vert u_{1}\right\Vert \cdots\left\Vert u_{m}\right\Vert \text{.}
\end{equation*}
\end{description}

\noindent When all the components $\mathcal{M}_{m}\left( E_{1},\ldots
,E_{m};F\right) $ are complete under the (quasi-) norm $\left\Vert \cdot
\right\Vert _{\mathcal{M}}$ above, $\mathcal{M}$ is said to be a \textit{%
(quasi-) Banach ideal of multilinear mappings. }For a fixed ideal of
multilinear mappings $\mathcal{M}$ and a positive integer $m\in \mathbb{N}$,
the class%
\begin{equation*}
\mathcal{M}_{m}:=\underset{E_{1},\ldots ,E_{m},F}{\bigcup }\mathcal{M}%
_{m}\left( E_{1},\ldots ,E_{m};F\right)
\end{equation*}%
is called an \textit{ideal of }$m$\textit{-linear mappings}.
\end{definition}

For the homogeneous polynomials we have the following concepts:

\begin{definition}
Let $E,F$ be normed spaces. A polynomial $P\in $ $\mathcal{P}\left(
^{m}E;F\right) $ is said to be of\ \textit{finite type} if there exists $%
k\in \mathbb{N},~\varphi _{i}\in E^{\prime }$ and $b_{i}\in F$ for $%
i=1,\ldots ,k$, such that%
\begin{equation*}
P\left( x\right) =\overset{k}{\underset{i=1}{\sum }}\varphi _{i}\left(
x\right) ^{m}b_{i}\text{.}
\end{equation*}%
The subspace of all finite-type members of $\mathcal{P}\left( ^{m}E;F\right) 
$ is denoted by $\mathcal{P}_{f}\left( ^{m}E;F\right) $.
\end{definition}

\begin{definition}[Polynomial ideal \protect\cite{pi2}]
\label{polyidealdef}For each positive integers $m$, let $\mathcal{P}_{m}$
denote the class of all continuous $m$-homogeneous polynomials between
Banach spaces. A \textit{polynomial ideal} $\mathcal{Q}$ (or \textit{ideal
of homogeneous polynomials}) is a subclass of the class $\mathcal{P}=%
\underset{m=1}{\overset{\infty }{\bigcup }}\mathcal{P}_{m}$ of all
continuous homogeneous polynomials between Banach spaces such that for all $%
m\in \mathbb{N}$ and all Banach spaces $E$ and $F$, the components%
\begin{equation*}
\mathcal{Q}_{m}\left( ^{m}E;F\right) :=\mathcal{P}_{m}\left( ^{m}E;F\right)
\cap \mathcal{Q}
\end{equation*}%
satisfy:

\begin{description}
\item[(Pa)] $\mathcal{Q}_{m}\left( ^{m}E;F\right) $ is a linear subspace of $%
\mathcal{P}_{m}\left( ^{m}E;F\right) $ which contains the finity-type $m$%
-homogeneous polynomials;

\item[(Pb)] the ideal property: if $u\in\mathcal{L}_{1}\left( G;E\right) $ $%
P\in\mathcal{Q}_{m}\left( ^{m}E;F\right) $ and $t\in\mathcal{L}_{1}\left(
F;H\right) $, then%
\begin{equation*}
t\circ P\circ u\in\mathcal{Q}_{m}\left( ^{m}G;H\right) .
\end{equation*}
\end{description}

\noindent Moreover, $\mathcal{Q}$ is said to be a \textit{(quasi-) normed
polynomial ideal} if there exists a \linebreak map $\left\Vert \cdot
\right\Vert _{\mathcal{Q}}:\mathcal{Q}\rightarrow \lbrack 0,\infty )$
satisfying:

\begin{description}
\item[(P1)] $\left\Vert \cdot\right\Vert _{\mathcal{Q}}$ restricted to $%
\mathcal{Q}_{m}\left( ^{m}E;F\right) $ is a (quasi-) norm, for all $m\in%
\mathbb{N}$ and all Banach spaces $E$ and $F$;

\item[(P2)] $\left\Vert id_{m}:\mathbb{K}\rightarrow \mathbb{K}:id_{m}\left(
\lambda \right) =\lambda ^{m}\right\Vert _{\mathcal{Q}}=1$, for all $m\in 
\mathbb{N}$;

\item[(P3)] If $u\in\mathcal{L}_{1}\left( G;E\right) ,~P\in\mathcal{Q}%
_{m}\left( ^{m}E;F\right) $ and $t\in\mathcal{L}_{1}\left( F;H\right) $, then%
\begin{equation*}
\left\Vert t\circ P\circ u\right\Vert _{\mathcal{Q}}\leq\left\Vert
t\right\Vert \left\Vert P\right\Vert _{\mathcal{Q}}\left\Vert u\right\Vert
^{m}\text{.}
\end{equation*}
\end{description}

\noindent When all the components $\mathcal{Q}_{m}\left( ^{m}E;F\right) $
are complete under the (quasi-) norm $\left\Vert \cdot \right\Vert _{%
\mathcal{Q}}$ above, then $\mathcal{Q}$ is called a \textit{(quasi-) Banach
polynomial ideal. }For a fixed polynomial ideal $\mathcal{Q}$ and a positive
integer $m\in \mathbb{N}$, the class%
\begin{equation*}
\mathcal{Q}_{m}:=\underset{E,F}{\bigcup }\mathcal{Q}_{m}\left( ^{m}E;F\right)
\end{equation*}%
is called an \textit{ideal of }$m$\textit{-homogeneous polynomials}.
\end{definition}

\section{Basic results}

A fact apparently overloked in the literature is that every $m$-linear
operator is in fact a polynomial (we thank Prof. Pilar Rueda and R. Aron for
important conversations about it). More precisely, if 
\begin{equation*}
T:E_{1}\times \cdots \times E_{m}\rightarrow F
\end{equation*}%
is an $m$-linear operator then, denoting $E:=E_{1}\times \cdots \times
E_{m}, $ the map%
\begin{eqnarray*}
P &:&E\rightarrow F \\
P(x_{1},\ldots ,x_{m}) &=&T(x_{1},\ldots ,x_{m})
\end{eqnarray*}%
is an $m$-homogeneous polynomial. This fact can be easily proved by using
tensor products. So, one can wonder why to define separately ideals of
polynomials and ideals of multilinear operators, having in mind that every $%
m $-linear operator is in fact an $m$-homogeneous polynomial. Well, we can
give a couple of reasons for that. A more obvious reason is that when
considering a multilinear operator as a polynomial we have%
\begin{equation*}
\left\Vert T(x_{1},\ldots ,x_{m})\right\Vert \leq \left\Vert T\right\Vert
\left\Vert (x_{1},\ldots ,x_{m})\right\Vert ^{m},
\end{equation*}%
and this estimate is less precise than%
\begin{equation}
\left\Vert T(x_{1},\ldots ,x_{m})\right\Vert \leq \left\Vert T\right\Vert
\left\Vert x_{1}\right\Vert \cdots \left\Vert x_{m}\right\Vert .  \label{g66}
\end{equation}%
So, one cannot unify the theory of polynomial ideals and multilinear ideals
just by realizing that multilinear operators are in fact homogeneous
polynomials. In this section we unify the theory of polynomials and
multilinear operators in a more careful analytic viewpoint. More precisely,
if $m$ is a given positive integer and $n_{1},...,n_{m}$ are positive
integers such that $n_{1}+\cdots +n_{j}=$ $m$ we introduce the notion of $%
\left( n_{1},...,n_{j}\right) $-homogeneous polynomial for $j\in \{1,...,m\}$%
. When $j=1$ we have an $m$-homogeneous polynomial and when $j=m$ then we
have an $m$-linear operator. This kind of maps will be called
multipolynomials.

\begin{definition}
\label{def1}Let $m\in \mathbb{N}$ and $\left( n_{1},\ldots ,n_{m}\right) \in 
\mathbb{N}^{m}$. A mapping $P:E_{1}\times \cdots \times E_{m}\rightarrow F$
is said to be an $\left( n_{1},\ldots ,n_{m}\right) $-\textit{homogeneous
polynomial} if, for each $i=1,\ldots ,m$, the mapping \linebreak $P\left(
x_{1},\ldots ,x_{i-1},\cdot ,x_{i+1},\ldots ,x_{m}\right) :E_{i}\rightarrow
F $ is an $n_{i}$-homogeneous polynomial for every \linebreak $x_{1}\in
E_{1},\ldots ,x_{i-1}\in E_{i-1},x_{i+1}\in E_{i+1},\ldots ,x_{m}\in E_{m}$
fixed.
\end{definition}

We shall denote by $\mathcal{P}_{a}\left( ^{n_{1}}E_{1},\ldots
,^{n_{m}}E_{m};F\right) $ the vector space of all $\left( n_{1},\ldots
,n_{m}\right) $-homogeneous polynomials from the cartesian product $%
E_{1}\times \cdots \times E_{m}$ into $F$. We shall represent by \linebreak $%
\mathcal{P}\left( ^{n_{1}}E_{1},\ldots ,^{n_{m}}E_{m};F\right) $ the
subspace of all continuous members of $\mathcal{P}_{a}\left(
^{n_{1}}E_{1},\ldots ,^{n_{m}}E_{m};F\right) $. For each $P\in \mathcal{P}%
_{a}\left( ^{n_{1}}E_{1},\ldots ,^{n_{m}}E_{m};F\right) $ we shall set%
\begin{equation*}
\left\Vert P\right\Vert :=\sup \left\{ \left\Vert P\left( x_{1},\ldots
,x_{m}\right) \right\Vert ;x_{i}\in E_{i},\underset{i}{\max }\left\Vert
x_{i}\right\Vert _{E_{i}}\leq 1\right\} .
\end{equation*}

Our first step is to present a comprehensive list conditions which
characterize the continuous multipolynomials similarly as in (\ref{g66}).

Henceforth $\mathcal{L}_{a}^{s}\left( ^{m}E;F\right) $ denotes the space of
all $m$-linear forms from $E\times \cdots \times E$ to $F$ which are
symmetric and for each $A\in \mathcal{L}_{a}\left( ^{m}E;F\right) $ the $m$%
-homogeneous polynomial $\widehat{A}\in \mathcal{P}_{a}\left( ^{m}E;F\right) 
$ is defined by $\widehat{A}\left( x\right) =Ax^{m}$ for every $x\in E$.

We recall some results from the theory of homogeneous polynomials between
Banach spaces that will be useful in this paper (these results can be found,
for instance, in \cite[Theorem 1.10]{Muj}, \cite[Theorem 2.2]{Muj} and \cite[%
Corollary 2.3]{Muj}).

\begin{itemize}
\item (Polarization Formula) If $A\in \mathcal{L}_{a}^{s}\left(
^{m}E;F\right) $, then for $x_{0},\ldots ,x_{m}\in E$ we have 
\begin{equation*}
A\left( x_{1},\ldots ,x_{m}\right) =\frac{1}{m!2^{m}}\underset{\varepsilon
_{j}=\pm 1}{\sum }\varepsilon _{1}\cdots \varepsilon _{m}A\left(
x_{0}+\varepsilon _{1}x_{1}+\cdots +\varepsilon _{m}x_{m}\right) ^{m}\text{.}
\end{equation*}

\item The mapping $A\mapsto \widehat{A}$ induces a vector space isomorphism
between $\mathcal{L}_{a}^{s}\left( ^{m}E;F\right) $ and $\mathcal{P}%
_{a}\left( ^{m}E;F\right) $.

\item We have the inequalities%
\begin{equation}
\left\Vert \widehat{A}\right\Vert \leq \left\Vert A\right\Vert \leq \frac{%
m^{m}}{m!}\left\Vert \widehat{A}\right\Vert \text{,}  \label{estre}
\end{equation}%
for every $A\in \mathcal{L}_{a}^{s}\left( ^{m}E;F\right) $.

\item A polynomial $P\in \mathcal{P}_{a}\left( ^{m}E;F\right) $ is
continuous if and only if $\left\Vert P\right\Vert <\infty $.

\item $\mathcal{P}\left( ^{m}E;F\right) $ is a Banach space under the norm $%
P\mapsto \left\Vert P\right\Vert $.

\item The mapping $A\mapsto \widehat{A}$ induces a topological isomorphism
between $\mathcal{L}^{s}\left( ^{m}E;F\right) $ and $\mathcal{P}\left(
^{m}E;F\right) $.
\end{itemize}

We begin with an useful lemma:

\begin{lemma}
\label{genlem2.5}Let $E_{1},\ldots ,E_{m},F$ be normed spaces and $P\in 
\mathcal{P}_{a}\left( ^{n_{1}}E_{1},\ldots ,^{n_{m}}E_{m};F\right) $. If $P$
is bounded by $C$ on an open ball $B_{E_{1}\times \cdots \times E_{m}}\left(
\left( a_{1},\ldots ,a_{m}\right) ;r\right) $ then $P$ is bounded by $C\frac{%
n_{1}^{n_{1}}}{n_{1}!}\cdots \frac{n_{m}^{n_{m}}}{n_{m}!}$ on the ball $%
B_{E_{1}\times \cdots \times E_{m}}\left( \left( 0,\ldots ,0\right)
;r\right) $.
\end{lemma}

\begin{proof}
Let $\left( x_{1},\ldots ,x_{m}\right) \in B_{E_{1}\times \cdots \times
E_{m}}\left( \left( 0,\ldots ,0\right) ;r\right) $. We prove this by
induction on $m$. When $m=1$ it is just the already well-known result from
the theory of homogeneous polynomials between Banach spaces (see \cite[Lemma
2.5]{Muj}). If $m>1$ then the multipolynomial $P\left( \underset{m-1}{%
\underbrace{\cdot ,\ldots ,\cdot }},y\right) \in \mathcal{P}_{a}\left(
^{n_{1}}E_{1},\ldots ,^{n_{m-1}}E_{m-1};F\right) $ is bounded by $C$ on the
ball \linebreak $B_{E_{1}\times \cdots \times E_{m-1}}\left( \left(
a_{1},\ldots ,a_{m-1}\right) ;r\right) $, for all $y\in B_{E_{m}}\left(
a_{m};r\right) $. The induction hypothesis implies that $P\left( \underset{%
m-1}{\underbrace{\cdot ,\ldots ,\cdot }},y\right) $ is bounded by $C\frac{%
n_{1}^{n_{1}}}{n_{1}!}\cdots \frac{n_{m-1}^{n_{m-1}}}{n_{m-1}!}$ on the ball 
$B_{E_{1}\times \cdots \times E_{m-1}}\left( \left( 0,\ldots ,0\right)
;r\right) $, whenever $y\in B_{E_{m}}\left( a_{m};r\right) $. We also have
from \cite[Theorem 2.2]{Muj} that the $n_{m}$-linear mapping, denoted by
\linebreak $A\left( x_{1},\ldots ,x_{m-1}\right) ,$ associated to the
polynomial $P\left( x_{1},\ldots ,x_{m-1},\cdot \right) \in \mathcal{P}%
_{a}\left( ^{n_{m}}E_{m};F\right) $ can be taken symmetrical. Now applying
the Polarization Formula \cite[Theorem 1.10]{Muj} to $A\left( x_{1},\ldots
,x_{m-1}\right) $ with $x_{0}=a_{m}$ and $x_{1}=\cdots =x_{n_{m}}=\frac{x_{m}%
}{n_{m}}$ we get%
\begin{align*}
& \left\Vert P\left( x_{1},\ldots ,x_{m}\right) \right\Vert \\
& =n_{m}^{n_{m}}\left\Vert A\left( x_{1},\ldots ,x_{m-1}\right) \left( \frac{%
x_{m}}{n_{m}}\right) ^{n_{m}}\right\Vert \\
& =n_{m}^{n_{m}}\left\Vert \frac{1}{n_{m}!2^{n_{m}}}\underset{\varepsilon
_{j}=\pm 1}{\sum }\varepsilon _{1}\cdots \varepsilon _{n_{m}}A\left(
x_{1},\ldots ,x_{m-1}\right) \left( a_{m}+\left( \varepsilon _{1}+\cdots
+\varepsilon _{n_{m}}\right) \frac{x_{m}}{n_{m}}\right) ^{n_{m}}\right\Vert
\\
& \leq \frac{n_{m}^{n_{m}}}{n_{m}!2^{n_{m}}}\underset{\varepsilon _{j}=\pm 1}%
{\sum }\left\Vert A\left( x_{1},\ldots ,x_{m-1}\right) \left( a_{m}+\left(
\varepsilon _{1}+\cdots +\varepsilon _{n_{m}}\right) \frac{x_{m}}{n_{m}}%
\right) ^{n_{m}}\right\Vert \\
& =\frac{n_{m}^{n_{m}}}{n_{m}!2^{n_{m}}}\underset{\varepsilon _{j}=\pm 1}{%
\sum }\left\Vert P\left( \underset{m-1}{\underbrace{\cdot ,\ldots ,\cdot }}%
,a_{m}+\left( \varepsilon _{1}+\cdots +\varepsilon _{n_{m}}\right) \frac{%
x_{m}}{n_{m}}\right) \left( x_{1},\ldots ,x_{m-1}\right) \right\Vert \\
& \leq \frac{n_{m}^{n_{m}}}{n_{m}!2^{n_{m}}}2^{n_{m}}C\frac{n_{1}^{n_{1}}}{%
n_{1}!}\cdots \frac{n_{m-1}^{n_{m-1}}}{n_{m-1}!} \\
& =C\frac{n_{1}^{n_{1}}}{n_{1}!}\cdots \frac{n_{m-1}^{n_{m-1}}}{n_{m-1}!}%
\frac{n_{m}^{n_{m}}}{n_{m}!}\text{.}
\end{align*}%
Then it follows that $P$ is bounded by $C\frac{n_{1}^{n_{1}}}{n_{1}!}\cdots 
\frac{n_{m}^{n_{m}}}{n_{m}!}$ on the ball $B_{E_{1}\times \cdots \times
E_{m}}\left( \left( 0,\ldots ,0\right) ;r\right) $, and the proof is
complete.
\end{proof}

Now we are ready to characterize continuous multipolynomials.

\begin{theorem}
\label{16}Let $E_{1},\ldots ,E_{m},F$ be normed spaces and $P\in \mathcal{P}%
_{a}\left( ^{n_{1}}E_{1},\ldots ,^{n_{m}}E_{m};F\right) $. The following
conditions are equivalent:

\begin{description}
\item[(i)] $P$ is continuous;

\item[(ii)] $P$ is continuous at the origin;

\item[(iii)] There exists a constant $C\geq 0$ such that%
\begin{equation*}
\left\Vert P\left( x_{1},\ldots ,x_{m}\right) \right\Vert \leq C\left\Vert
x_{1}\right\Vert ^{n_{1}}\cdots \left\Vert x_{m}\right\Vert ^{n_{m}}\text{,}
\end{equation*}%
for all $\left( x_{1},\ldots ,x_{m}\right) \in E_{1}\times \cdots \times
E_{m}$;

\item[(iv)] $\left\Vert P\right\Vert <\infty $;

\item[(v)] $P$ is uniformly continuous on bounded subsets of $E_{1}\times
\cdots \times E_{m}$;

\item[(vi)] $P$ is bounded on every ball with finite radius;

\item[(vii)] $P$ is bounded on some ball;

\item[(viii)] $P$ is bounded on some ball with center at the origin.
\end{description}
\end{theorem}

\begin{proof}
The implications $\left( i\right) \Rightarrow (ii)$ and $\left( vi\right)
\Rightarrow \left( vii\right) $ are obvious.

$\left( ii\right) \Rightarrow \left( iii\right) $: Suppose $P$ continuous at
the origin. Then, there exist $\delta >0$ such that%
\begin{equation*}
\left( x_{1},\ldots ,x_{m}\right) \in E_{1}\times \cdots \times
E_{m},~\left\Vert \left( x_{1},\ldots ,x_{m}\right) \right\Vert <\delta
\Rightarrow \left\Vert P\left( x_{1},\ldots ,x_{m}\right) \right\Vert <1%
\text{.}
\end{equation*}%
The inequality in $\left( iii\right) $ is obvious if $x_{i}=0$ for some $%
i=1,\ldots ,m$. So we can assume $x_{i}\neq 0$ for all $i=1,\ldots ,m$. Then,%
\begin{equation*}
\left\Vert \left( \frac{\delta x_{1}}{2\left\Vert x_{1}\right\Vert },\ldots ,%
\frac{\delta x_{m}}{2\left\Vert x_{m}\right\Vert }\right) \right\Vert =\frac{%
\delta }{2}<\delta
\end{equation*}%
and thus%
\begin{align*}
\left\Vert P\left( x_{1},\ldots ,x_{m}\right) \right\Vert & =\left( \frac{2}{%
\delta }\right) ^{n_{1}+\cdots +n_{m}}\left\Vert x_{1}\right\Vert
^{n_{1}}\cdots \left\Vert x_{m}\right\Vert ^{n_{m}}\left\Vert P\left( \frac{%
\delta x_{1}}{2\left\Vert x_{1}\right\Vert },\ldots ,\frac{\delta x_{m}}{%
2\left\Vert x_{m}\right\Vert }\right) \right\Vert \\
& <\left( \frac{2}{\delta }\right) ^{n_{1}+\cdots +n_{m}}\left\Vert
x_{1}\right\Vert ^{n_{1}}\cdots \left\Vert x_{m}\right\Vert ^{n_{m}}\text{.}
\end{align*}%
This give us $\left( iii\right) $ with $C=\left( \frac{2}{\delta }\right)
^{n_{1}+\cdots +n_{m}}$.

$\left( iii\right) \Rightarrow \left( iv\right) $: If $\left( iii\right) $
is true then we have in particular,%
\begin{equation*}
\left\Vert P\left( x_{1},\ldots ,x_{m}\right) \right\Vert \leq C\left\Vert
x_{1}\right\Vert ^{n_{1}}\cdots \left\Vert x_{m}\right\Vert ^{n_{m}}\leq C%
\text{,}
\end{equation*}%
for all $x_{1}\in E_{1},\ldots ,x_{m}\in E_{m}$, with $\left\Vert
x_{1}\right\Vert ,\ldots ,\left\Vert x_{m}\right\Vert \leq 1$. This shows
that $\left\Vert P\right\Vert \leq C$.

$\left( iv\right) \Rightarrow \left( v\right) $: Let $a=\left( a_{1},\ldots
,a_{m}\right) ,~x=\left( x_{1},\ldots ,x_{m}\right) \in E_{1}\times \cdots
\times E_{m}$ with 
\begin{equation*}
\underset{i}{\max }\left\Vert x_{i}\right\Vert \leq r
\end{equation*}%
~and~ 
\begin{equation*}
\underset{i}{\max }\left\Vert a_{i}\right\Vert \leq r.
\end{equation*}%
Then%
\begin{eqnarray}
\left\Vert P\left( a_{1},\ldots ,a_{i-1},\cdot ,x_{i+1},\ldots ,x_{m}\right)
\right\Vert &\leq &\left\Vert a_{1}\right\Vert ^{n_{1}}\cdots \left\Vert
a_{i-1}\right\Vert ^{n_{i-1}}\left\Vert x_{i+1}\right\Vert ^{n_{i+1}}\cdots
\left\Vert x_{m}\right\Vert ^{n_{m}}\left\Vert P\right\Vert  \label{22} \\
&\leq &r^{n_{1}+\cdots +n_{i-1}+n_{i+1}+\cdots +n_{m}}\left\Vert
P\right\Vert \text{.}  \notag
\end{eqnarray}%
Let $A\left( a_{1},\ldots ,a_{i-1},x_{i+1},\ldots ,x_{m}\right) $ be the
symmetric $n_{i}$-linear form associated to \linebreak $P\left( a_{1},\ldots
,a_{i-1},\cdot ,x_{i+1},\ldots ,x_{m}\right) .$ From (\ref{estre}) we get%
\begin{eqnarray}
\left\Vert A\left( a_{1},\ldots ,a_{i-1},x_{i+1},\ldots ,x_{m}\right)
\right\Vert &\leq &\frac{n_{i}^{n_{i}}}{n_{i}!}\left\Vert P\left(
a_{1},\ldots ,a_{i-1},\cdot ,x_{i+1},\ldots ,x_{m}\right) \right\Vert
\label{23} \\
&\leq &\frac{n_{i}^{n_{i}}}{n_{i}!}r^{n_{1}+\cdots +n_{i-1}+n_{i+1}+\cdots
+n_{m}}\left\Vert P\right\Vert \text{,}  \notag
\end{eqnarray}%
for every $i=1,\ldots ,m$. Since $\left\Vert P\right\Vert <\infty $ the
inequalities (\ref{22}) and (\ref{23}) show us that both the $i$-th ordinary
polynomial $P\left( a_{1},\ldots ,a_{i-1},\cdot ,x_{i+1},\ldots
,x_{m}\right) \in $ $\mathcal{P}_{a}\left( ^{n_{i}}E_{i};F\right) $ as its
associated multilinear mapping $A\left( a_{1},\ldots ,a_{i-1},x_{i+1},\ldots
,x_{m}\right) \in \mathcal{L}_{a}\left( ^{n_{i}}E_{i};F\right) $ are
continuous for each $i=1,\ldots ,m$. Now we can write%
\begin{align*}
& \left\Vert P\left( x\right) -P\left( a\right) \right\Vert \\
& \leq \overset{m}{\underset{i=1}{\sum }}\left\Vert P\left( a_{1},\ldots
,a_{i-1},x_{i},\ldots ,x_{m}\right) -P\left( a_{1},\ldots
,a_{i},x_{i+1},\ldots ,x_{m}\right) \right\Vert \\
& =\overset{m}{\underset{i=1}{\sum }}\left\Vert A\left( a_{1},\ldots
,a_{i-1},x_{i+1},\ldots ,x_{m}\right) \left( x_{i}\right) ^{n_{i}}-A\left(
a_{1},\ldots ,a_{i-1},x_{i+1},\ldots ,x_{m}\right) \left( a_{i}\right)
^{n_{i}}\right\Vert \\
& \leq \overset{m}{\underset{i=1}{\sum }}\left( \left\Vert 
\begin{array}{c}
A\left( a_{1},\ldots ,a_{i-1},x_{i+1},\ldots ,x_{m}\right) \left(
x_{i}-a_{i},x_{i},\ldots ,x_{i}\right) +\cdots \\ 
\cdots +A\left( a_{1},\ldots ,a_{i-1},x_{i+1},\ldots ,x_{m}\right) \left(
a_{i},\ldots ,a_{i},x_{i}-a_{i}\right)%
\end{array}%
\right\Vert \right) \\
& \leq \overset{m}{\underset{i=1}{\sum }}n_{i}\left\Vert A\left(
a_{1},\ldots ,a_{i-1},x_{i+1},\ldots ,x_{m}\right) \right\Vert \left\Vert
x_{i}-a_{i}\right\Vert r^{n_{i}-1} \\
& \leq \overset{m}{\underset{i=1}{\sum }}n_{i}\left( \frac{n_{i}^{n_{i}}}{%
n_{i}!}r^{n_{1}+\cdots +n_{i-1}+n_{i+1}+\cdots +n_{m}}\left\Vert
P\right\Vert \right) \left\Vert x-a\right\Vert r^{n_{i}-1} \\
& \leq \left( \overset{m}{\underset{i=1}{\sum }}\frac{n_{i}^{n_{i}+1}}{n_{i}!%
}\right) r^{n_{1}+\cdots +n_{m}-1}\left\Vert P\right\Vert \left\Vert
x-a\right\Vert \text{,}
\end{align*}%
and the uniform continuity of $P$ on bounded subsets of $E_{1}\times \cdots
\times E_{m}$ follows.

$\left( v\right) \Rightarrow \left( i\right) $: Let us show that $P$ is
continuous at an arbitrary point $a=\left( a_{1},\ldots ,a_{m}\right) \in
E=E_{1}\times \cdots \times E_{m}$. Given $\varepsilon >0$ it follows from
the uniform continuity of $P$ on bounded subsets that there exist $\delta
_{0}>0$ such that, for every $x=\left( x_{1},\ldots ,x_{m}\right) ,y=\left(
y_{1},\ldots ,y_{m}\right) \in B_{E}\left( 0;\left\Vert a\right\Vert
+1\right) $,%
\begin{equation*}
\left\Vert x-y\right\Vert <\delta _{0}\Rightarrow \left\Vert P\left(
x\right) -P\left( y\right) \right\Vert <\varepsilon \text{.}
\end{equation*}%
Defining $\delta =\min \left\{ \delta _{0},1\right\} $ we get%
\begin{equation*}
x\in E_{1}\times \cdots \times E_{m},~~\left\Vert x-a\right\Vert <\delta
\Rightarrow \left\Vert P\left( x\right) -P\left( a\right) \right\Vert
<\varepsilon
\end{equation*}%
and thus $P$ is continuous at the point $a$.

$\left( iii\right) \Rightarrow \left( vi\right) $: Let $B$ be a ball with
center at $\left( a_{1},\ldots ,a_{m}\right) \in E_{1}\times \cdots \times
E_{m}$ and radius $r>0$. For every $\left( x_{1},\ldots ,x_{m}\right) \in B$
the hypothesis $\left( iii\right) $ give us a constant $C\geq 0$ such that%
\begin{equation*}
\left\Vert P\left( x_{1},\ldots ,x_{m}\right) \right\Vert \leq C\left\Vert
x_{1}\right\Vert ^{n_{1}}\cdots \left\Vert x_{m}\right\Vert ^{n_{m}}\leq
C(r+\left\Vert a_{1}\right\Vert )^{n_{1}}\cdots \left( r+\left\Vert
a_{m}\right\Vert \right) ^{n_{m}}
\end{equation*}%
and so $P$ is bounded on $B$.

$\left( vii\right) \Rightarrow \left( viii\right) $: It follows immediately
from the Lemma \ref{genlem2.5}.

$\left( viii\right) \Rightarrow \left( iv\right) $: Suppose that there exist 
$r>0$ ~and~ $C\geq 0$ such that%
\begin{equation*}
\left\Vert P\left( x_{1},\ldots ,x_{m}\right) \right\Vert \leq C,~~~~\forall
\left( x_{1},\ldots ,x_{m}\right) \in B_{E_{1}\times \cdots \times
E_{m}}\left( \left( 0,\ldots ,0\right) ;r\right) \text{.}
\end{equation*}%
Thus, given $x_{1}\in E_{1},\ldots ,x_{m}\in E_{m}$, with $\left\Vert
x_{1}\right\Vert ,\ldots ,\left\Vert x_{m}\right\Vert \leq 1$, we have $%
\left( \frac{r}{2}x_{1},\ldots ,\frac{r}{2}x_{m}\right) \in B_{E_{1}\times
\cdots \times E_{m}}\left( \left( 0,\ldots ,0\right) ;r\right) $ and hence%
\begin{equation*}
\left\Vert P\left( x_{1},\ldots ,x_{m}\right) \right\Vert =\left( \frac{2}{r}%
\right) ^{n_{1}+\cdots +n_{m}}\left\Vert P\left( \frac{r}{2}x_{1},\ldots ,%
\frac{r}{2}x_{m}\right) \right\Vert \leq C\left( \frac{2}{r}\right)
^{n_{1}+\cdots +n_{m}}\text{.}
\end{equation*}
\end{proof}

\begin{corollary}
Let $E_{1},\ldots ,E_{m}$ and $F$ be normed spaces and let $P$ be a
continuous multipolynomial in $\mathcal{P}\left( ^{n_{1}}E_{1},\ldots
,^{n_{m}}E_{m};F\right) $. Then:

\begin{description}
\item[(i)] $\left\Vert P\left( x_{1},\ldots ,x_{m}\right) \right\Vert \leq
\left\Vert P\right\Vert \left\Vert x_{1}\right\Vert ^{n_{1}}\cdots
\left\Vert x_{m}\right\Vert ^{n_{m}}$, for all $\left( x_{1},\ldots
,x_{m}\right) \in E_{1}\times \cdots \times E_{m}$.

\item[(ii)] $\left\Vert P\right\Vert =\inf \left\{ C\geq 0;\left\Vert
P\left( x_{1},\ldots ,x_{m}\right) \right\Vert \leq C\left\Vert
x_{1}\right\Vert ^{n_{1}}\cdots \left\Vert x_{m}\right\Vert ^{n_{m}}\text{
for all }x_{j}\in E_{j},j=1,\ldots ,m\right\} $.
\end{description}
\end{corollary}

\begin{remark}
The following results are also ensured:

\begin{itemize}
\item Still in the context of normed spaces, the pointwise convergence of a
sequence $\left( P_{j}\right) $ in $\mathcal{P}_{a}\left(
^{n_{1}}E_{1},\ldots ,^{n_{m}}E_{m};F\right) $ implies that the limit
mapping $P$ is also a multipolynomial in that same space. More can be said
if we take Banach spaces into account.

\item The mapping $P\mapsto \left\Vert P\right\Vert $ defines a norm on the
vector space $\mathcal{P}\left( ^{n_{1}}E_{1},\ldots ,^{n_{m}}E_{m};F\right) 
$ which turns it into a Banach space, provided that $F$ is complete.

\item If $F$ is just a normed space and all the factor spaces of the domain
of $P$ are Banach spaces, then continuity and separated continuity are the
same. Besides, the Uniform Boundedness Principle (UBP) and the
Banach--Steinhaus Theorem (BST) hold, as we see next.
\end{itemize}
\end{remark}

\begin{theorem}[Uniform Boundedness Principle]
\label{plumultpol}Let $E_{1},\ldots ,E_{m}$ be Banach spaces, $F$ be a
normed space and let $\left\{ P_{i}\right\} _{i\in I}$ be a family in $%
\mathcal{P}\left( ^{n_{1}}E_{1},\ldots ,^{n_{m}}E_{m};F\right) $. The
following conditions are equivalent:

\begin{description}
\item[(i)] For every $x=\left( x_{1},\ldots ,x_{m}\right) \in E_{1}\times
\cdots \times E_{m}$ there exists $C_{x}<\infty $ such that%
\begin{equation*}
\underset{i\in I}{\sup }\left\Vert P_{i}\left( x\right) \right\Vert <C_{x}%
\text{.}
\end{equation*}

\item[(ii)] The family $\left\{ P_{i}\right\} _{i\in I}$ is norm bounded,
that is,%
\begin{equation*}
\underset{i\in I}{\sup }\left\Vert P_{i}\right\Vert <\infty \text{.}
\end{equation*}
\end{description}
\end{theorem}

\begin{corollary}[Banach-Steinhaus Theorem]
\label{bstmultpol}Let $\left( P_{j}\right) _{j=1}^{\infty }$ be a sequence
in $\mathcal{P}\left( ^{n_{1}}E_{1},\ldots ,^{n_{m}}E_{m};F\right) $ such
that $\left( P_{j}\left( x_{1},\ldots ,x_{m}\right) \right) _{j=1}^{\infty }$
is convergent in $F$ for all $\left( x_{1},\ldots ,x_{m}\right) \in
E_{1}\times \cdots \times E_{m}$. If we define%
\begin{equation*}
P:E_{1}\times \cdots \times E_{m}\rightarrow F
\end{equation*}%
by 
\begin{equation*}
P\left( x_{1},\ldots ,x_{m}\right) :=\underset{j\rightarrow \infty }{\lim }%
P_{j}\left( x_{1},\ldots ,x_{m}\right) \text{,}
\end{equation*}%
then $P\in \mathcal{P}\left( ^{n_{1}}E_{1},\ldots ,^{n_{m}}E_{m};F\right) $.
\end{corollary}

Under the same hypothesis of the BST presented above, we can also conclude
that $P_{j}$ converges to $P$ uniformly on compact subsets of $E_{1}\times
\cdots \times E_{m}$.

It is worth noting that the extreme cases $n_{1}=\ldots =n_{m}=1$ or $m=1$
recover the corresponding classical theorems for multilinear mappings (see,
for instance, \cite{sandberg, bern}) and homogeneous polynomials. The
special case $m=n_{1}=1$ recovers the classicals UBP and BST for linear
operators.

\section{The unified approach: multipolynomial ideals}

In this section, we present a more general and unified version for the
ideals of polynomials and multilinear operators.

\begin{definition}
Let $E_{1},\ldots ,E_{m},F$ be normed spaces. A multipolynomial \linebreak $%
P\in \mathcal{P}\left( ^{n_{1}}E_{1},\ldots ,^{n_{m}}E_{m};F\right) $ is
said to be of\ \textit{finite type} if there exists $k\in \mathbb{N}%
,~\varphi _{i}^{\left( j\right) }\in E_{j}^{\prime }$ and $b_{i}\in F$ with $%
i=1,\ldots ,k$ and $j=1,\ldots ,m$, such that%
\begin{equation*}
P\left( x_{1},\ldots ,x_{m}\right) =\overset{k}{\underset{i=1}{\sum }}%
\varphi _{i}^{\left( 1\right) }\left( x_{1}\right) ^{n_{1}}\cdots \varphi
_{i}^{\left( m\right) }\left( x_{m}\right) ^{n_{m}}b_{i}\text{.}
\end{equation*}%
We shall represent by $\mathfrak{F}\left( ^{n_{1}}E_{1},\ldots
,^{n_{m}}E_{m};F\right) $ the subspace of all finite type members of
\linebreak $\mathcal{P}\left( ^{n_{1}}E_{1},\ldots ,^{n_{m}}E_{m};F\right) $%
. The so-called \textit{approximable multipolynomials} are defined to be the
members of 
\begin{equation*}
\mathfrak{P}_{A}\left( ^{n_{1}}E_{1},\ldots ,^{n_{m}}E_{m};F\right) :=%
\overline{\mathfrak{F}\left( ^{n_{1}}E_{1},\ldots ,^{n_{m}}E_{m};F\right) }%
\text{.}
\end{equation*}
\end{definition}

\begin{definition}[Multipolynomial ideal]
\label{multipolyidealdef}For each $m\in \mathbb{N}$ and multi-degree $\left(
n_{1},\ldots ,n_{m}\right) \in \mathbb{N}^{m}$, let $\mathcal{P}_{m}^{\left(
n_{1},\ldots ,n_{m}\right) }$ denote the class of all continuous $\left(
n_{1},\ldots ,n_{m}\right) $-homogeneous polynomials between Banach spaces. A%
\textit{\ multipolynomial ideal} $\mathfrak{U}$ is a subclass of the class $%
\mathfrak{P}:\mathcal{=}\underset{m=1}{\overset{\infty }{\bigcup }}\left( 
\underset{\left( n_{1},\ldots ,n_{m}\right) \in \mathbb{N}^{m}}{\bigcup }%
\mathcal{P}_{m}^{\left( n_{1},\ldots ,n_{m}\right) }\right) $ of all
continuous multipolynomials between Banach spaces such that for all $m\in 
\mathbb{N}$, multi-degree $\left( n_{1},\ldots ,n_{m}\right) \in \mathbb{N}%
^{m}$ and all Banach spaces $E_{1},\ldots ,E_{m}$ and $F$, the components%
\begin{equation*}
\mathfrak{U}_{m}^{\left( n_{1},\ldots ,n_{m}\right) }\left(
^{n_{1}}E_{1},\ldots ,^{n_{m}}E_{m};F\right) :=\mathcal{P}\left(
^{n_{1}}E_{1},\ldots ,^{n_{m}}E_{m};F\right) \cap \mathfrak{U}
\end{equation*}%
satisfy:

\begin{description}
\item[(Ua)] $\mathfrak{U}_{m}^{\left( n_{1},\ldots ,n_{m}\right) }\left(
^{n_{1}}E_{1},\ldots ,^{n_{m}}E_{m};F\right) $ is a linear subspace of $%
\mathcal{P}\left( ^{n_{1}}E_{1},\ldots ,^{n_{m}}E_{m};F\right) $ which
contains the $\left( n_{1},\ldots ,n_{m}\right) $-homogeneous polynomials of
finite type;

\item[(Ub)] the ideal property: if $P\in \mathfrak{U}_{m}^{\left(
n_{1},\ldots ,n_{m}\right) }\left( ^{n_{1}}E_{1},\ldots
,^{n_{m}}E_{m};F\right) ,~u_{j}\in \mathcal{L}_{1}\left( G_{j};E_{j}\right) $
for $j=1,\ldots ,m$, and $t\in \mathcal{L}_{1}\left( F;H\right) $, then%
\begin{equation*}
t\circ P\circ \left( u_{1},\ldots ,u_{m}\right) \in \mathfrak{U}_{m}^{\left(
n_{1},\ldots ,n_{m}\right) }\left( ^{n_{1}}G_{1},\ldots
,^{n_{m}}G_{m};H\right) .
\end{equation*}
\end{description}

\noindent Moreover, $\mathfrak{U}$ is said to be a \textit{(quasi-) normed
multipolynomial ideal} if there exists a map $\left\Vert \cdot \right\Vert _{%
\mathfrak{U}}:\mathfrak{U}\rightarrow \lbrack 0,\infty )$ satisfying:

\begin{description}
\item[(U1)] $\left\Vert \cdot \right\Vert _{\mathfrak{U}}$ restricted to $%
\mathfrak{U}_{m}^{\left( n_{1},\ldots ,n_{m}\right) }\left(
^{n_{1}}E_{1},\ldots ,^{n_{m}}E_{m};F\right) $ is a (quasi-) norm, for all $%
m\in \mathbb{N}$, multi-degree $\left( n_{1},\ldots ,n_{m}\right) \in $ $%
\mathbb{N}^{m}$ and all Banach spaces $E_{1},\ldots ,E_{m}$ and $F$;

\item[(U2)] $\left\Vert id_{m}^{\left( n_{1},\ldots ,n_{m}\right) }:\mathbb{K%
}^{m}\rightarrow \mathbb{K}:id_{m}^{\left( n_{1},\ldots ,n_{m}\right)
}\left( \lambda _{1},\ldots ,\lambda _{m}\right) =\lambda _{1}^{n_{1}}\cdots
\lambda _{m}^{n_{m}}\right\Vert _{\mathfrak{U}}=1$, for all $m\in \mathbb{N}$
and $\left( n_{1},\ldots ,n_{m}\right) \in $ $\mathbb{N}^{m}$;

\item[(U3)] If $P\in \mathfrak{U}_{m}^{\left( n_{1},\ldots ,n_{m}\right)
}\left( ^{n_{1}}E_{1},\ldots ,^{n_{m}}E_{m};F\right) ,~u_{j}\in \mathcal{L}%
_{1}\left( G_{j};E_{j}\right) $ for $j=1,\ldots ,m$, and $t\in \mathcal{L}%
_{1}\left( F;H\right) $, then%
\begin{equation*}
\left\Vert t\circ P\circ \left( u_{1},\ldots ,u_{m}\right) \right\Vert _{%
\mathfrak{U}}\leq \left\Vert t\right\Vert \left\Vert P\right\Vert _{%
\mathfrak{U}}\left\Vert u_{1}\right\Vert ^{n_{1}}\cdots \left\Vert
u_{m}\right\Vert ^{n_{m}}\text{.}
\end{equation*}
\end{description}

\noindent When all the components $\mathfrak{U}_{m}^{\left( n_{1},\ldots
,n_{m}\right) }\left( ^{n_{1}}E_{1},\ldots ,^{n_{m}}E_{m};F\right) $ are
complete under the (quasi-) norm $\left\Vert \cdot \right\Vert _{\mathfrak{U}%
}$ above, then $\mathfrak{U}$ is called a \textit{(quasi-) Banach
multipolynomial ideal. }For a fixed multipolynomial ideal $\mathfrak{U}$, a
positive integer $m\in \mathbb{N}$ and a multi-degree $\left( n_{1},\ldots
,n_{m}\right) \in \mathbb{N}^{m}$, the class%
\begin{equation*}
\mathfrak{U}_{m}^{\left( n_{1},\ldots ,n_{m}\right) }:=\underset{%
E_{1},\ldots ,E_{m},F}{\bigcup }\mathfrak{U}_{m}^{\left( n_{1},\ldots
,n_{m}\right) }\left( ^{n_{1}}E_{1},\ldots ,^{n_{m}}E_{m};F\right)
\end{equation*}%
is called an \textit{ideal of }$\left( n_{1},\ldots ,n_{m}\right) $\textit{%
-homogeneous polynomials}.

\noindent A multipolynomial ideal $\mathfrak{U}$ is said to be \textit{closed%
} if all components\textit{\ }$\mathfrak{U}_{m}^{\left( n_{1},\ldots
,n_{m}\right) }\left( ^{n_{1}}E_{1},\ldots ,^{n_{m}}E_{m};F\right) $ are
closed subspace of $\left( \mathcal{P}\left( ^{n_{1}}E_{1},\ldots
,^{n_{m}}E_{m};F\right) ,\left\Vert \cdot \right\Vert \right) $.\bigskip
\end{definition}

Natural examples can be found.

\begin{definition}
Given a multipolynomial ideal $\mathfrak{U}$, we shall define%
\begin{equation*}
\overline{\mathfrak{U}}_{m}^{\left( n_{1},\ldots ,n_{m}\right) }\left(
^{n_{1}}E_{1},\ldots ,^{n_{m}}E_{m};F\right) :=\overline{\mathfrak{U}%
_{m}^{\left( n_{1},\ldots ,n_{m}\right) }\left( ^{n_{1}}E_{1},\ldots
,^{n_{m}}E_{m};F\right) }^{\left\Vert \cdot \right\Vert }\text{,}
\end{equation*}%
for all $m\in \mathbb{N}$, multi-degree $\left( n_{1},\ldots ,n_{m}\right)
\in $ $\mathbb{N}^{m}$ and all Banach spaces $E_{1},\ldots ,E_{m}$ and $F$.
\end{definition}

\begin{definition}
A multipolynomial $P\in \mathcal{P}\left( ^{n_{1}}E_{1},\ldots
,^{n_{m}}E_{m};F\right) $ is said to be \textit{compact} (resp. \textit{%
weakly compact}) if%
\begin{equation*}
P\left( B_{E_{1}}\times \cdots \times B_{E_{m}}\right)
\end{equation*}%
is relatively compact (resp. weakly relatively compact). In that case we
write $\mathfrak{K}\left( ^{n_{1}}E_{1},\ldots ,^{n_{m}}E_{m};F\right) $
(resp. $\mathfrak{W}\left( ^{n_{1}}E_{1},\ldots ,^{n_{m}}E_{m};F\right) $).
\end{definition}

\begin{definition}
Let $0<p,q_{1},\ldots ,q_{m}$ and let $E_{1},\ldots ,E_{m}$ and $F$ be
Banach spaces. A continuous $\left( n_{1},\ldots ,n_{m}\right) $-homogeneous
polynomial $P:E_{1}\times \cdots \times E_{m}\rightarrow F$ is said to be 
\textit{absolutely }$\left( p;q_{1},\ldots ,q_{m}\right) $-\textit{summing }%
if there exists $C\geq 0$ such that%
\begin{equation}
\left( \overset{\infty }{\underset{j=1}{\dsum }}\left\Vert P\left(
x_{j}^{\left( 1\right) },\ldots ,x_{j}^{\left( m\right) }\right) \right\Vert
^{p}\right) ^{\frac{1}{p}}\leq C\overset{m}{\underset{k=1}{\dprod }}%
\left\Vert \left( x_{j}^{\left( k\right) }\right) _{j=1}^{\infty
}\right\Vert _{w,q_{k}}^{n_{k}}  \label{matosabs}
\end{equation}%
for all $\left( x_{j}^{\left( k\right) }\right) _{j=1}^{\infty }\in \ell
_{q_{k}}^{w}\left( E_{k}\right) $, with$~k=1,\ldots ,m$. In this case we
write $\mathfrak{P}_{\text{\emph{as}}\left( p,q_{1},\ldots ,q_{m}\right)
}\left( ^{n_{1}}E_{1},\ldots ,^{n_{m}}E_{m};F\right) $.

In a quite more demanding way, $P$ is said to be \textit{fully} $\left(
p;q_{1},\ldots ,q_{m}\right) $-\textit{summing }if there exists\textit{\ }$%
C\geq 0$ such that\textit{\ }%
\begin{equation}
\left( \overset{\infty }{\underset{j_{1},\ldots ,j_{m}=1}{\dsum }}\left\Vert
P\left( x_{j_{1}}^{\left( 1\right) },\ldots ,x_{j_{m}}^{\left( m\right)
}\right) \right\Vert ^{p}\right) ^{\frac{1}{p}}\leq C\overset{m}{\underset{%
k=1}{\dprod }}\left\Vert \left( x_{j}^{\left( k\right) }\right)
_{j=1}^{\infty }\right\Vert _{w,q_{k}}^{n_{k}}  \label{matosfully}
\end{equation}%
for all $\left( x_{j}^{\left( k\right) }\right) _{j=1}^{\infty }\in \ell
_{q_{k}}^{w}\left( E_{k}\right) $, with$~k=1,\ldots ,m$. In this case we
write $\mathfrak{P}_{\text{\emph{fs}}\left( p,q_{1},\ldots ,q_{m}\right)
}\left( ^{n_{1}}E_{1},\ldots ,^{n_{m}}E_{m};F\right) $.
\end{definition}

The infimum of the $C>0$ for which inequality (\ref{matosabs}) (resp. (\ref%
{matosfully})) always holds defines a norm for the case $p\geq 1$ and a $p$%
-norm for $p<1$ on the space $\mathfrak{P}_{as\left( p,q_{1},\ldots
,q_{m}\right) }^{\left( n_{1},\ldots ,n_{m}\right) }\left(
^{n_{1}}E_{1},\ldots ,^{n_{m}}E_{m};F\right) $ (resp. $\mathfrak{P}_{\text{%
\emph{fs}}\left( p,q_{1},\ldots ,q_{m}\right) }\left( ^{n_{1}}E_{1},\ldots
,^{n_{m}}E_{m};F\right) $). In any case, we thus obtain complete topological
metrizable spaces.

\bigskip

We now give a list of several examples.

\begin{description}
\item[$\mathfrak{P}$] Ideal of continuous multipolynomials;

\item[$\mathfrak{F}$] Ideal of finite-type multipolynomials;

\item[$\overline{\mathfrak{U}}$] The closure of a multipolynomial ideal $%
\mathfrak{U}$;

\item[$\mathfrak{P}_{A}$] Ideal of approximable multipolynomials;

\item[$\mathfrak{K}$] Ideal of compact multipolynomials;

\item[$\mathfrak{W}$] Ideal of weakly compact multipolynomials;

\item[$\mathfrak{P}_{\text{\emph{as}}}$] Ideal of absolutely summing
multipolynomials;

\item[$\mathfrak{P}_{\text{\emph{fs}}}$] Ideal of fully summing
multipolynomials.
\end{description}

\bigskip

The multipolynomial ideals $\mathfrak{P}$, $\overline{\mathfrak{U}}$, $%
\mathfrak{P}_{A}$, $\mathfrak{K}$, $\mathfrak{W}$ are closed. If $\frac{1}{p}%
\leq \frac{n_{1}}{q_{1}}+\cdots +\frac{n_{m}}{q_{m}}$ then $\mathfrak{P}_{%
\emph{as}}$ and $\mathfrak{P}_{\emph{fs}}$ are Banach multipolynomial ideals
for $p\geq 1$ and $p$-Banach multipolynomial ideals for $p<1$. We have
null-spaces when $\frac{1}{p}>\frac{n_{1}}{q_{1}}+\cdots +\frac{n_{m}}{q_{m}}
$.

\begin{remark}
We shall recall that as particular cases or, more precisely, as extreme
cases ($n_{1}=\cdots =n_{m}=1$ or $m=1$), every ideal of multilinear
mappings (which includes the linear operator ideals) as well as every
polynomial ideal already established in the literature is a multipolynomial
ideal. They will be called \textit{extreme multipolynomial ideals}.
\end{remark}

\section{Multipolynomial hyper-ideals}

Recently in papers \cite{bo9} and \cite{bo10} the authors introduced and
developed the respective notions of hyper-ideals of multilinear operators
and homogeneous polynomials between Banach spaces. While the well studied
notions of ideals of multilinear operators (multi-ideals) as well as
polynomial ideals relies on the composition with linear operators (the
so-called ideal property), the notion proposed by the authors, called now as
hyper-ideal property, considers in \cite{bo9} the compositions with
multilinear operators and, under the polynomial viewpoint, considers in \cite%
{bo10} the compositions with homogeneous polynomials. Historically speaking,
the hyper-ideal property has already been studied individually for some
specific classes, see, e.g, \cite{dps,popa1,popa2}, and then \cite{bo9,bo10}
started the systematic study of the classes satisfying this stronger
condition. The aim of this section is to invoke the multipolynomials again,
as we have done before, to generalize and propose a unified approach for all
these isolated notions of hyper-ideals of operators (multilinear and
polynomial) which have been studied separately so far.

\begin{definition}[Hyper-ideal of multilinear operators \protect\cite{bo9}]
\label{ha}A \textit{hyper-ideal of multilinear operators} is a subclass $%
\mathcal{H}$ of the class of all continuous multilinear operators between
Banach spaces such that for all $n\in\mathbb{N}$ and all Banach spaces $%
E_{1},\ldots,E_{n}$ and $F$, the components%
\begin{equation*}
\mathcal{H}\left( E_{1},\ldots,E_{n};F\right) :=\mathcal{L}\left(
E_{1},\ldots,E_{n};F\right) \cap\mathcal{H}
\end{equation*}
satisfy:

\begin{description}
\item[(ha)] $\mathcal{H}\left( E_{1},\ldots,E_{n};F\right) $ is a linear
subspace of $\mathcal{L}\left( E_{1},\ldots,E_{n};F\right) $ which contains
the $n$-linear operators of finite type ;

\item[(hb)] The \textbf{hyper-ideal property}: given natural numbers $n$ and 
$1\leq m_{1}<\cdots <m_{n}$ and Banach spaces $G_{1},\ldots
,G_{m_{n}},E_{1},\ldots ,E_{n},F$ and $H$, if $B_{1}\in \mathcal{L}\left(
G_{1},\ldots ,G_{m_{1}};E_{1}\right) ,\ldots ,$\linebreak $B_{n}\in \mathcal{%
L}\left( G_{m_{n-1}+1},\ldots ,G_{m_{n}};E_{n}\right) ,A\in \mathcal{H}%
\left( E_{1},\ldots ,E_{n};F\right) $ and $t\in \mathcal{L}\left( F;H\right) 
$, then%
\begin{equation*}
t\circ A\circ \left( B_{1},\ldots ,B_{n}\right) \in \mathcal{H}\left(
G_{1},\ldots ,G_{m_{n}};H\right) \text{.}
\end{equation*}
\end{description}

\noindent Moreover, $\mathcal{H}$ is said to be a \textit{(quasi-) normed
hyper-ideal of multilinear operators} if there exists a map $\left\Vert
\cdot\right\Vert _{\mathcal{H}}:\mathcal{H}\rightarrow\lbrack0,\infty)$
satisfying:

\begin{description}
\item[(h1)] $\left\Vert \cdot \right\Vert _{\mathcal{H}}$ restricted to $%
\mathcal{H}\left( E_{1},\ldots ,E_{n};F\right) $ is a (quasi-) norm, for all 
$n\in \mathbb{N}$ and all Banach spaces $E_{1},\ldots ,E_{n}$ and $F$;

\item[(h2)] $\left\Vert I_{n}:\mathbb{K}^{n}\rightarrow\mathbb{K}%
,I_{n}\left( \lambda_{1},\ldots,\lambda_{n}\right) =\lambda_{1}\cdots\lambda
_{n}\right\Vert _{\mathcal{H}}=1$, for all $n\in\mathbb{N}$;

\item[(h3)] The \textbf{hyper-ideal inequality}: if $B_{1}\in \mathcal{L}%
\left( G_{1},\ldots ,G_{m_{1}};E_{1}\right) ,\ldots ,$\linebreak $B_{n}\in 
\mathcal{L}\left( G_{m_{n-1}+1},\ldots ,G_{m_{n}};E_{n}\right) ,A\in 
\mathcal{H}\left( E_{1},\ldots ,E_{n};F\right) $ and $t\in \mathcal{L}\left(
F;H\right) $, then%
\begin{equation*}
\left\Vert t\circ A\circ \left( B_{1},\ldots ,B_{n}\right) \right\Vert _{%
\mathcal{H}}\leq \left\Vert t\right\Vert \left\Vert A\right\Vert _{\mathcal{H%
}}\left\Vert B_{1}\right\Vert \cdots \left\Vert B_{n}\right\Vert \text{.}
\end{equation*}
\end{description}

\noindent When all the components $\mathcal{H}\left( E_{1},\ldots
,E_{n};F\right) $ are complete under the (quasi-) norm $\left\Vert
\cdot\right\Vert _{\mathcal{H}}$ above, then $\left( \mathcal{H},\left\Vert
\cdot\right\Vert _{\mathcal{H}}\right) $ is called a \textit{(quasi-) Banach
hyper-ideal of multilinear operators.}
\end{definition}

It is plain that every (normed, quasi-normed, Banach, quasi-Banach)
hyper-ideal is a (normed, quasi-normed, Banach, quasi-Banach) multi-ideal.

\begin{definition}[Polynomial hyper-ideal \protect\cite{bo10}]
\label{pa}A \textit{polynomial} \textit{hyper-ideal} is a subclass $\mathcal{%
Q}$ of the class of all continuous homogeneous polynomials between Banach
spaces such that for all $n\in\mathbb{N}$ and all Banach spaces $E$ and $F$,
the components%
\begin{equation*}
\mathcal{Q}\left( ^{n}E;F\right) :=\mathcal{P}\left( ^{n}E;F\right) \cap%
\mathcal{Q}
\end{equation*}
satisfy:

\begin{description}
\item[(pa)] $\mathcal{Q}\left( ^{n}E;F\right) $ is a linear subspace of $%
\mathcal{P}\left( ^{n}E;F\right) $ which contains the $n$-homogeneous
polynomials of finite type;

\item[(pb)] The \textbf{hyper-ideal property}: given $m,n\in\mathbb{N}$ and
Banach spaces $E,F,G$ and $H$, if $Q\in\mathcal{P}\left( ^{m}G;E\right) ,P\in%
\mathcal{Q}\left( ^{n}E;F\right) $ and $t\in\mathcal{L}\left( F;H\right) $,
then%
\begin{equation*}
t\circ P\circ Q\in\mathcal{Q}\left( ^{mn}G;H\right) \text{.}
\end{equation*}
\end{description}

\noindent If there exist a map $\left\Vert \cdot\right\Vert _{\mathcal{Q}}:%
\mathcal{Q}\rightarrow\lbrack0,\infty)$ and a sequence $\left( C_{j}\right)
_{j=1}^{\infty}$ of real numbers with $C_{j}\geq1$ for every $j\in\mathbb{N}$
and $C_{1}=1$, such that:

\begin{description}
\item[(p1)] $\left\Vert \cdot\right\Vert _{\mathcal{Q}}$ restricted to $%
\mathcal{Q}\left( ^{n}E;F\right) $ is a (quasi-) norm, for all $n\in\mathbb{N%
}$ and all Banach spaces $E$ and $F$;

\item[(p2)] $\left\Vert I_{n}:\mathbb{K}\rightarrow\mathbb{K},I_{n}\left(
\lambda\right) =\lambda^{n}\right\Vert _{\mathcal{Q}}=1$, for all $n\in%
\mathbb{N}$;

\item[(p3)] The \textbf{hyper-ideal inequality:} if $Q\in\mathcal{P}\left(
^{m}G;E\right) ,P\in\mathcal{Q}\left( ^{n}E;F\right) $ and $t\in \mathcal{L}%
\left( F;H\right) $, then%
\begin{equation*}
\left\Vert t\circ P\circ Q\right\Vert _{\mathcal{Q}}\leq\left( C_{m}\right)
^{n}\left\Vert t\right\Vert \left\Vert P\right\Vert _{\mathcal{Q}}\left\Vert
Q\right\Vert ^{n}\text{,}
\end{equation*}
\end{description}

\noindent then $\left( \mathcal{Q},\left\Vert \cdot\right\Vert _{\mathcal{Q}%
}\right) $ is called a \textit{(quasi-) normed polynomial }$\left(
C_{j}\right) _{j=1}^{\infty}$\textit{-hyper-ideal}. When all the components $%
\mathcal{Q}\left( ^{n}E;F\right) $ are complete under the (quasi-) norm $%
\left\Vert \cdot\right\Vert _{\mathcal{Q}}$ above, then $\left( \mathcal{Q}%
,\left\Vert \cdot\right\Vert _{\mathcal{Q}}\right) $ is called a \textit{%
(quasi-) Banach polynomial }$\left( C_{j}\right) _{j=1}^{\infty}$-\textit{%
hyper-ideal.}
\end{definition}

When $C_{j}=1$ for every $j\in \mathbb{N}$, we simply say that $\mathcal{Q}$
is a \textit{(quasi-) normed/(quasi-) Banach polynomial hyper-ideal. }When
the hyper-ideal property (and inequality) holds for every\textit{\ }$n\in 
\mathbb{N}$, but only for $m=1$, we say that $\mathcal{Q}$ is a \textit{%
(quasi-) normed/(quasi-) Banach polynomial ideal }(remember that $C_{1}=1$).

\begin{definition}[Polynomial two-sided ideal \protect\cite{bo10}]
\label{ts}A \textit{polynomial} \textit{two}-\textit{sided\ ideal} is a
subclass $\mathcal{Q}$ of the class of all continuous homogeneous
polynomials between Banach spaces such that for all $n\in \mathbb{N}$ and
all Banach spaces $E$ and $F$, the components%
\begin{equation*}
\mathcal{Q}\left( ^{n}E;F\right) :=\mathcal{P}\left( ^{n}E;F\right) \cap 
\mathcal{Q}
\end{equation*}%
satisfy:

\begin{description}
\item[(ts-a)] $\mathcal{Q}\left( ^{n}E;F\right) $ is a linear subspace of $%
\mathcal{P}\left( ^{n}E;F\right) $ which contains the $n$-homogeneous
polynomials of finite type;

\item[(ts-b)] The \textbf{two-sided ideal property:} given $m,n,r\in \mathbb{%
N}$ and Banach spaces $E,F,G$ and $H$, if $Q\in\mathcal{P}\left(
^{m}G;E\right) ,P\in\mathcal{Q}\left( ^{n}E;F\right) $ and $R\in \mathcal{P}%
\left( ^{r}F;H\right) $, then%
\begin{equation*}
R\circ P\circ Q\in\mathcal{Q}\left( ^{mnr}G;H\right) \text{.}
\end{equation*}
\end{description}

\noindent If there exist a map $\left\Vert \cdot\right\Vert _{\mathcal{Q}}:%
\mathcal{Q}\rightarrow\lbrack0,\infty)$ and a sequence $\left(
C_{j},K_{j}\right) _{j=1}^{\infty}$ of pairs of real numbers with $%
C_{j},K_{j}\geq1$ for every $j\in\mathbb{N}$ and $C_{1}=K_{1}=1$, such that:

\begin{description}
\item[(ts-1)] $\left\Vert \cdot\right\Vert _{\mathcal{Q}}$ restricted to $%
\mathcal{Q}\left( ^{n}E;F\right) $ is a (quasi-) norm, for all $n\in\mathbb{N%
}$ and all Banach spaces $E$ and $F$;

\item[(ts-2)] $\left\Vert I_{n}:\mathbb{K}\rightarrow\mathbb{K},I_{n}\left(
\lambda\right) =\lambda^{n}\right\Vert _{\mathcal{Q}}=1$, for all $n\in%
\mathbb{N}$;

\item[(ts-3)] The \textbf{two-sided ideal inequality}: if $Q\in\mathcal{P}%
\left( ^{m}G;E\right) ,P\in\mathcal{Q}\left( ^{n}E;F\right) $ and $R\in%
\mathcal{P}\left( ^{r}F;H\right) $, then%
\begin{equation*}
\left\Vert R\circ P\circ Q\right\Vert _{\mathcal{Q}}\leq K_{r}\left(
C_{m}\right) ^{rn}\left\Vert R\right\Vert \left\Vert P\right\Vert _{\mathcal{%
Q}}\left\Vert Q\right\Vert ^{rn}\text{,}
\end{equation*}
\end{description}

\noindent then $\left( \mathcal{Q},\left\Vert \cdot \right\Vert _{\mathcal{Q}%
}\right) $ is called a \textit{(quasi-) normed polynomial }$\left(
C_{j},K_{j}\right) _{j=1}^{\infty }$\textit{-two-sided ideal}. When all the
components $\mathcal{Q}\left( ^{n}E;F\right) $ are complete under the
(quasi-) norm $\left\Vert \cdot \right\Vert _{\mathcal{Q}}$ above, then $%
\left( \mathcal{Q},\left\Vert \cdot \right\Vert _{\mathcal{Q}}\right) $ is
called a \textit{(quasi-) Banach polynomial }$\left( C_{j},K_{j}\right)
_{j=1}^{\infty }$-\textit{two-sided\ ideal.}
\end{definition}

When $C_{j}=K_{j}=1$ for every $j\in\mathbb{N}$, we simply say that $%
\mathcal{Q}$ is a \textit{(quasi-) normed/(quasi-) Banach polynomial
two-sided ideal.}

\begin{remark}
The condition $C_{1}=K_{1}=1$ guarantees that every (normed, quasi-normed,
Banach, quasi-Banach) polynomial $\left( C_{j},K_{j}\right) _{j=1}^{\infty }$%
-two-sided ideal is a (normed, quasi-normed, Banach, quasi-Banach)
polynomial $\left( C_{j}\right) _{j=1}^{\infty }$-hyper-ideal; and that, as
we mentioned before, every (normed, quasi-normed, Banach, quasi-Banach)
polynomial $\left( C_{j}\right) _{j=1}^{\infty }$-hyper-ideal is a (normed,
quasi-normed, Banach, quasi-Banach) polynomial ideal.
\end{remark}

Next we extend the above notions to the multipolynomials.

\begin{definition}[Multipolynomial hyper-ideal]
A \textit{hyper-ideal of multipolynomials} (or \textit{multipolynomial
hyper-ideals}) is a subclass $\mathfrak{H}$ of the class of all continuous
multipolynomials between Banach spaces such that for all $n\in \mathbb{N}$,
multi-degree $\left( k_{1},\ldots ,k_{n}\right) \in \mathbb{N}^{n}$ and all
Banach spaces $E_{1},\ldots ,E_{n}$ and $F$, the components%
\begin{equation*}
\mathfrak{H}_{n}^{\left( k_{1},\ldots ,k_{n}\right) }\left(
^{k_{1}}E_{1},\ldots ,^{k_{n}}E_{n};F\right) :=\mathcal{P}\left(
^{k_{1}}E_{1},\ldots ,^{k_{n}}E_{n};F\right) \cap \mathfrak{H}
\end{equation*}%
satisfy:

\begin{description}
\item[(Ha)] $\mathfrak{H}_{n}^{\left( k_{1},\ldots ,k_{n}\right) }\left(
^{k_{1}}E_{1},\ldots ,^{k_{n}}E_{n};F\right) $ is a linear subspace of $%
\mathcal{P}\left( ^{k_{1}}E_{1},\ldots ,^{k_{n}}E_{n};F\right) $ which
contains the $\left( k_{1},\ldots ,k_{n}\right) $-homogeneous polynomials of
finite type ;

\item[(Hb)] The \textbf{hyper-ideal property}: given natural numbers $n,~~$ $%
1\leq m_{1}<\cdots <m_{n}$ and $r_{1},\ldots ,r_{m_{n}},k_{1},\ldots ,k_{n}$
and $r$ and Banach spaces $G_{1},\ldots ,G_{m_{n}},E_{1},\ldots ,E_{n},F$
and $H$, if \linebreak $Q_{1}\in \mathcal{P}\left( ^{r_{1}}G_{1},\ldots
,^{r_{m_{1}}}G_{m_{1}};E_{1}\right) ,\ldots ,Q_{n}\in \mathcal{P}\left(
^{r_{m_{n-1}+1}}G_{m_{n-1}+1},\ldots ,^{r_{m_{n}}}G_{m_{n}};E_{n}\right) ,$%
\linebreak $P\in \mathfrak{H}_{n}^{\left( k_{1},\ldots ,k_{n}\right) }\left(
^{k_{1}}E_{1},\ldots ,^{k_{n}}E_{n};F\right) $ and $R\in \mathcal{P}\left(
^{r}F;H\right) $, then%
\begin{equation*}
R\circ P\circ \left( Q_{1},\ldots ,Q_{n}\right) \in \mathfrak{H}%
_{m_{n}}^{\left( r_{1}k_{1}r,\ldots ,r_{m_{1}}k_{1}r,\ldots
,r_{m_{n}}k_{n}r\right) }\left( ^{r_{1}k_{1}r}G_{1},\ldots
,^{r_{m_{n}}k_{n}r}G_{m_{n}};H\right) .
\end{equation*}
\end{description}

\noindent If there exist a map $\left\Vert \cdot \right\Vert _{\mathfrak{H}}:%
\mathfrak{H}\rightarrow \lbrack 0,\infty )$ and a sequence $\left(
C_{j},K_{j}\right) _{j=1}^{\infty }$ of pairs of real numbers with $%
C_{j},K_{j}\geq 1$ for every $j\in \mathbb{N}$ and $C_{1}=K_{1}=1$, such
that:

\begin{description}
\item[(H1)] $\left\Vert \cdot \right\Vert _{\mathfrak{H}}$ restricted to $%
\mathfrak{H}_{n}^{\left( k_{1},\ldots ,k_{n}\right) }\left(
^{k_{1}}E_{1},\ldots ,^{k_{n}}E_{n};F\right) $ is a (quasi-) norm, for all $%
n\in \mathbb{N}$, multi-degree $\left( k_{1},\ldots ,k_{n}\right) \in 
\mathbb{N}^{n}$ and all Banach spaces $E_{1},\ldots ,E_{n}$ and $F$;

\item[(H2)] $\left\Vert I_{n}^{\left( k_{1},\ldots ,k_{n}\right) }:\mathbb{K}%
^{n}\rightarrow \mathbb{K},I_{n}^{\left( k_{1},\ldots ,k_{n}\right) }\left(
\lambda _{1},\ldots ,\lambda _{n}\right) =\lambda _{1}^{k_{1}}\cdots \lambda
_{n}^{k_{n}}\right\Vert _{\mathfrak{H}}=1$, for all $n\in \mathbb{N}$ and $%
\left( k_{1},\ldots ,k_{n}\right) \in \mathbb{N}^{n}$;

\item[(H3)] The \textbf{hyper-ideal inequality}: if $Q_{1}\in \mathcal{P}%
\left( ^{r_{1}}G_{1},\ldots ,^{r_{m_{1}}}G_{m_{1}};E_{1}\right) ,\ldots ,$%
\linebreak $Q_{n}\in \mathcal{P}\left( ^{r_{m_{n-1}+1}}G_{m_{n-1}+1},\ldots
,^{r_{m_{n}}}G_{m_{n}};E_{n}\right) ,P\in \mathfrak{H}_{n}^{\left(
k_{1},\ldots ,k_{n}\right) }\left( ^{k_{1}}E_{1},\ldots
,^{k_{n}}E_{n};F\right) $ and \linebreak $R\in \mathcal{P}\left(
^{r}F;H\right) $, then%
\begin{align*}
& \left\Vert R\circ P\circ \left( Q_{1},\ldots ,Q_{n}\right) \right\Vert _{%
\mathfrak{H}} \\
& \leq K_{r}\left( C_{r_{1}}\cdots C_{r_{m_{1}}}\right) ^{rk_{1}}\cdots
\left( C_{r_{m_{n-1}+1}}\cdots C_{r_{m_{n}}}\right) ^{rk_{n}}\left\Vert
R\right\Vert \left\Vert P\right\Vert _{\mathfrak{H}}^{r}\left\Vert
Q_{1}\right\Vert ^{rk_{1}}\cdots \left\Vert Q_{n}\right\Vert ^{rk_{n}}\text{,%
}
\end{align*}
\end{description}

\noindent then $\left( \mathfrak{H},\left\Vert \cdot \right\Vert _{\mathfrak{%
H}}\right) $ is called a \textit{(quasi-) normed multipolynomial }$\left(
C_{j},K_{j}\right) _{j=1}^{\infty }$-\textit{hyper-ideal. }When all the
components $\mathfrak{H}_{n}^{\left( k_{1},\ldots ,k_{n}\right) }\left(
^{k_{1}}E_{1},\ldots ,^{k_{n}}E_{n};F\right) $ are complete under
the\linebreak (quasi-) norm $\left\Vert \cdot \right\Vert _{\mathfrak{H}}$
above, then $\left( \mathfrak{H},\left\Vert \cdot \right\Vert _{\mathfrak{H}%
}\right) $ is called a \textit{(quasi-) Banach multipolynomial }$\left(
C_{j},K_{j}\right) _{j=1}^{\infty }$-\textit{hyper-ideal.}
\end{definition}

When $C_{j}=K_{j}=1$ for every $j\in \mathbb{N}$, we simply say that $%
\mathfrak{H}$ is a \textit{(quasi-) normed/(quasi-) Banach multipolynomial
hyper-ideal.}

\begin{remark}
Our multipolynomial hyper-ideals definition is more general and unifying in
the sense that it recovers the multilinear and polynomial cases. Indeed,
setting $n=1=m_{1}$ we get Definition \ref{ts} if $r>1$ and Definition \ref%
{pa} if $r=1$. In the other end, setting $k_{1}=\cdots =k_{n}=r_{1}=\cdots
=r_{m_{n}}=r=1$ we recover Definition \ref{ha}. Finally, it is plain that
every (normed, quasi-normed, Banach, quasi-Banach) multipolynomial
hyper-ideal is a (normed, quasi-normed, Banach, quasi-Banach)
multipolynomial ideal.
\end{remark}

\section{A Bohnenblust--Hille inequality for multipolynomials}

In this section we present a unified version to the Bohnenblust--Hille
inequalities \cite{bh} for homogeneous polynomials and for multilinear
forms. The theory of Bohnenblust--Hille inequalities has been exhaustively
investigated in recent years (see, for instance \cite{albu, bohr, ncm, c,
diniz, dn, dn2, pt, san}, and the references therein).

Let $\alpha =(\alpha _{j})_{j=1}^{\infty }$ be a sequence in $\mathbb{N}\cup
\{0\}$ and, as usual, define $|\alpha |=\sum\limits_{j=1}^{\infty }\alpha
_{j};$ in this case we also denote $\mathbf{x}^{\alpha
}:=\prod\nolimits_{j}x_{j}^{\alpha _{j}}.$ An $m$-homogeneous polynomial $%
P:c_{0}\rightarrow \mathbb{K}$ $\left( \mathbb{K=R}\text{ or }\mathbb{C}%
\right) $ is denoted by%
\begin{equation*}
P(x)=\sum_{|\alpha |=m}c_{\alpha }(P)\mathbf{x}^{\alpha }.
\end{equation*}%
We recall that the norm of $P$ is given by $\Vert P\Vert :=\sup_{\left\Vert
x\right\Vert \leq 1}|P(x)|$.\bigskip

The Bohnenblust--Hille inequality for homogeneous polynomials \cite{bh}
asserts that

\begin{theorem}[Polynomial Bohnenblust--Hille inequality]
\label{bhpol}Let $m$ be a positive fixed integer. The following assertions
are equivalent:

\begin{description}
\item[(i)] There exists a constant $C_{\mathbb{K},m}\geq 1$ such that%
\begin{equation*}
\left( \underset{|\alpha |=m}{\sum }\left\vert c_{\alpha }\left( P\right)
\right\vert ^{p}\right) ^{\frac{1}{p}}\leq C_{\mathbb{K},m}\left\Vert
P\right\Vert
\end{equation*}%
for all continuous $m$-homogeneous polynomial $P:c_{0}\rightarrow \mathbb{K}$%
;

\item[(ii)] 
\begin{equation*}
p\geq \frac{2m}{m+1}\text{.}
\end{equation*}
\end{description}
\end{theorem}

\bigskip

We also have the Bohnenblust--Hille inequality for multilinear forms \cite%
{bh}:

\begin{theorem}[Multilinear Bohnenblust--Hille inequality]
\label{bhmultilin}Let $m$ be a positive fixed integer. The following
assertions are equivalent:

\begin{description}
\item[(i)] There exists a constant $C_{\mathbb{K},m}\geq 1$ such that%
\begin{equation*}
\left( \sum_{i_{1},\ldots ,i_{m}=1}^{\infty }\left\vert T\left(
e_{i_{1}},\ldots ,e_{i_{m}}\right) \right\vert ^{p}\right) ^{\frac{1}{p}%
}\leq C_{\mathbb{K},m}\left\Vert T\right\Vert
\end{equation*}%
for all continuous $m$-linear forms $T:c_{0}\times \cdots \times
c_{0}\rightarrow \mathbb{K}$;

\item[(ii)] 
\begin{equation*}
p\geq \frac{2m}{m+1}\text{.}
\end{equation*}
\end{description}
\end{theorem}

\bigskip

Our next step is to unify these results above. We shall observe that a $%
\left( n_{1},\ldots ,n_{m}\right) $-homogeneous polynomial $P:c_{0}\times
\cdots \times c_{0}\rightarrow \mathbb{K}$ can be allways written as%
\begin{equation*}
P\left( x^{\left( 1\right) },\ldots ,x^{\left( m\right) }\right) =\underset{%
\left\vert \alpha ^{\left( 1\right) }\right\vert =n_{1},\ldots ,\left\vert
\alpha ^{\left( m\right) }\right\vert =n_{m}}{\sum }c_{\alpha ^{\left(
1\right) }\ldots \alpha ^{\left( m\right) }}\left( P\right) \left( x^{\left(
1\right) }\right) ^{\alpha ^{\left( 1\right) }}\cdots \left( x^{\left(
m\right) }\right) ^{\alpha ^{\left( m\right) }}
\end{equation*}%
where, as we have previously defined, $\left( \alpha _{j}^{\left( i\right)
}\right) _{j=1}^{\infty }$ is a sequence in $\mathbb{N\cup }\left\{
0\right\} $, $\left\vert \alpha ^{\left( i\right) }\right\vert
=\sum\limits_{j=1}^{\infty }\alpha _{j}^{\left( i\right) }$ and $\left(
x^{\left( i\right) }\right) ^{\alpha ^{\left( i\right)
}}=\prod\nolimits_{j}\left( x_{j}^{\left( i\right) }\right) ^{\alpha
_{j}^{\left( i\right) }}$, for $i=1,\ldots ,m$.

Next we invoke the notion of multipolynomials to unify Theorems \ref{bhpol}
and \ref{bhmultilin}. In fact, these are respectively the particular cases $%
m=1$ and $n_{1}=\cdots =n_{m}=1$ of the following theorem:

\begin{theorem}[Multipolynomial Bohnenblust--Hille inequality]
Let $n_{1},\ldots ,n_{m}$ and $m$ be fixed positive integers. The following
assertions are equivalent:

\begin{description}
\item[(i)] There is a constant $C_{n_{1},\ldots ,n_{m}}\geq 1$ such that%
\begin{equation*}
\left( \underset{\left\vert \alpha ^{\left( 1\right) }\right\vert
=n_{1},\ldots ,\left\vert \alpha ^{\left( m\right) }\right\vert =n_{m}}{\sum 
}\left\vert c_{\alpha ^{\left( 1\right) }\ldots \alpha ^{\left( m\right)
}}\left( P\right) \right\vert ^{p}\right) ^{\frac{1}{p}}\leq C_{n_{1},\ldots
,n_{m}}\left\Vert P\right\Vert
\end{equation*}%
for all $\left( n_{1},\ldots ,n_{m}\right) $-homogeneous polynomial $%
P:c_{0}\times \cdots \times c_{0}\rightarrow \mathbb{K}$.

\item[(ii)] 
\begin{equation*}
p\geq \frac{2\left( \overset{m}{\underset{j=1}{\sum }}n_{j}\right) }{\left( 
\overset{m}{\underset{j=1}{\sum }}n_{j}\right) +1}\text{.}
\end{equation*}
\end{description}
\end{theorem}

\begin{proof}
$\left( ii\right) \Rightarrow \left( i\right) $: It suffices to prove the
assertion for 
\begin{equation*}
p_{0}=\frac{2\left( \overset{m}{\underset{j=1}{\sum }}n_{j}\right) }{\left( 
\overset{m}{\underset{j=1}{\sum }}n_{j}\right) +1}\text{.}
\end{equation*}%
Let $Q:c_{0}\rightarrow \mathbb{K}$ be the $\left( n_{1}+\cdots
+n_{m}\right) $-homogeneous polynomial given by%
\begin{equation*}
Q\left( z\right) :=P\left( \left( z_{j}\right) _{j\in \mathbb{N}_{1}},\ldots
,\left( z_{j}\right) _{j\in \mathbb{N}_{m}}\right) \text{,}
\end{equation*}%
where $\mathbb{N=N}_{1}\bigcup \cdots \bigcup \mathbb{N}_{m}$ is a disjoint
union with $card\left( \mathbb{N}_{j}\right) =card\left( \mathbb{N}\right) $%
, for $j=1,\dots ,m$. Note that since we are dealing with the $\sup $ norm
we have%
\begin{equation*}
\left\Vert Q\right\Vert \leq \left\Vert P\right\Vert
\end{equation*}%
and%
\begin{equation*}
\underset{\left\vert \beta \right\vert =n_{1}+\cdots +n_{m}}{\sum }%
\left\vert c_{\beta }\left( Q\right) \right\vert ^{p}=\underset{\left\vert
\alpha ^{\left( 1\right) }\right\vert =n_{1},\ldots ,\left\vert \alpha
^{\left( m\right) }\right\vert =n_{m}}{\sum }\left\vert c_{\alpha ^{\left(
1\right) }\ldots \alpha ^{\left( m\right) }}\left( P\right) \right\vert ^{p}%
\text{,}
\end{equation*}%
for all $p$. By the Polynomial Bohnenblust-Hille Inequality there exists a
constant $C_{n_{1}+\cdots +n_{m}}\geq 1$ such that%
\begin{eqnarray*}
\left( \underset{\left\vert \alpha ^{\left( 1\right) }\right\vert
=n_{1},\ldots ,\left\vert \alpha ^{\left( m\right) }\right\vert =n_{m}}{\sum 
}\left\vert c_{\alpha ^{\left( 1\right) }\ldots \alpha ^{\left( m\right)
}}\left( P\right) \right\vert ^{p_{0}}\right) ^{\frac{1}{p_{0}}} &=&\left( 
\underset{\left\vert \beta \right\vert =n_{1}+\cdots n_{m}}{\sum }\left\vert
c_{\beta }\left( Q\right) \right\vert ^{p_{0}}\right) ^{\frac{1}{p_{0}}} \\
&\leq &C_{n_{1}+\cdots +n_{m}}\left\Vert Q\right\Vert \\
&\leq &C_{n_{1}+\cdots +n_{m}}\left\Vert P\right\Vert \text{.}
\end{eqnarray*}

$\left( i\right) \Rightarrow \left( ii\right) $: Let 
\begin{eqnarray*}
T_{r} &:&c_{0}\times \cdots \times c_{0}\rightarrow \mathbb{K} \\
T_{r}\left( x^{\left( 1\right) },\ldots ,x^{\left( M\right) }\right)
&=&\sum_{i_{1},...,i_{M}=1}^{r}\pm x_{i_{1}}^{\left( 1\right) }\cdots
x_{i_{M}}^{\left( M\right) }
\end{eqnarray*}%
be the $M$-linear operator given by the Kahane-Salem-Zygmund inequality
(see, \cite[Lemma 6.1]{albu}) with%
\begin{equation*}
M=\overset{m}{\underset{j=1}{\sum }}n_{j}\text{.}
\end{equation*}%
Define $P_{r}:\overset{m}{\overbrace{c_{0}\times \cdots \times c_{0}}}%
\rightarrow \mathbb{K}$ by%
\begin{equation*}
P_{r}\left( x^{\left( 1\right) },\ldots ,x^{\left( m\right) }\right)
=T_{r}\left( \underset{n_{1}}{\underbrace{\left( x_{j}^{\left( 1\right)
}\right) _{j\in \mathbb{N}_{1}^{\left( 1\right) }},\ldots ,\left(
x_{j}^{\left( 1\right) }\right) _{j\in \mathbb{N}_{n_{1}}^{\left( 1\right) }}%
}},\ldots ,\underset{n_{m}}{\underbrace{\left( x_{j}^{\left( m\right)
}\right) _{j\in \mathbb{N}_{1}^{\left( m\right) }},\ldots ,\left(
x_{j}^{\left( m\right) }\right) _{j\in \mathbb{N}_{n_{m}}^{\left( m\right) }}%
}}\right) \text{,}
\end{equation*}%
where%
\begin{eqnarray*}
\mathbb{N} &\mathbb{=}&\mathbb{N}_{1}^{\left( 1\right) }\bigcup \cdots
\bigcup \mathbb{N}_{n_{1}}^{\left( 1\right) } \\
&&\vdots \\
\mathbb{N} &\mathbb{=}&\mathbb{N}_{1}^{\left( m\right) }\bigcup \cdots
\bigcup \mathbb{N}_{n_{m}}^{\left( m\right) }
\end{eqnarray*}%
are disjoint unions with $card\left( \mathbb{N}_{k}^{\left( i\right)
}\right) =card\left( \mathbb{N}\right) $, for $i=1,\dots ,m$ and $k=1,\ldots
,n_{i}$. Note that $P_{r}$ is an $\left( n_{1},\ldots ,n_{m}\right) $%
-homogeneous polynomial and $\left\Vert P_{r}\right\Vert \leq \left\Vert
T_{r}\right\Vert $.\ Moreover,%
\begin{eqnarray*}
\underset{\left\vert \alpha \right\vert =n_{1}+\cdots +n_{m}}{\sum }%
\left\vert c_{\alpha }\left( P_{r}\right) \right\vert ^{p}
&=&\sum_{i_{1},\ldots ,i_{M}=1}^{\infty }\left\vert T_{r}\left(
e_{i_{1}},\ldots ,e_{i_{M}}\right) \right\vert ^{p} \\
&=&r^{M}
\end{eqnarray*}%
for all $p$. Since%
\begin{equation*}
\left\Vert T_{r}\right\Vert \leq K_{M}r^{\frac{M+1}{2}}\text{,}
\end{equation*}%
by (i) we conclude that%
\begin{equation*}
\left( r^{M}\right) ^{\frac{1}{p}}\leq C_{n_{1},\ldots ,n_{m}}K_{M}r^{\frac{%
M+1}{2}}
\end{equation*}%
for all positive integers $r$. Thus%
\begin{equation*}
\frac{M}{p}\leq \frac{M+1}{2}
\end{equation*}%
and the proof is done.
\end{proof}

\end{document}